\documentclass[reqno,centertags,11pt,a4paper]{amsart}

\usepackage[a4paper]{geometry}
\usepackage{csquotes, amsmath, mathabx,amsthm, amssymb, ragged2e, hyperref, tikz}
\usepackage{cleveref}
\usepackage[english]{babel}
\graphicspath{ {figures/} }

\hypersetup{
    colorlinks=true,
    linkcolor=blue,
    filecolor=magenta,
    citecolor=blue,
    urlcolor=cyan,
    pdfpagemode=FullScreen,
}

\newtheorem{definition}{Definition}[section]
\newtheorem{theorem}[definition]{Theorem}

\newtheorem{corollary}[definition]{Corollary}
\newtheorem{lemma}[definition]{Lemma}

\begin{document}

\newcommand{\leqC}{\lesssim}
\newcommand{\geqC}{\gtrsim}
\newcommand{\abs}[1]{\left|{#1}\right|}
\newcommand{\norm}[1]{\left\|{#1}\right\|}
\newcommand{\sbk}[1]{\left({#1}\right)}
\newcommand{\mbk}[1]{\left[{#1}\right]}
\newcommand{\bbk}[1]{\left\{{#1}\right\}}
\newcommand{\LpNorm}[2]{\left\|{#1}\right\|_{L^{#2}}}
\newcommand{\LpNormOn}[3]{\left\|{#1}\right\|_{L^{#2}(#3)}}

\newcommand{\bR}[1]{\mathbb{R}^{#1}}
\newcommand{\fhalf}{\frac{1}{2}}

\newcommand{\GamAlR}{ \mathcal{P}(\delta) }
\newcommand{\GamHfR}{ \mathcal{P}(\delta^{\frac{\beta}{2}}) }

\newcommand{\MainThmLHSOn}[1]{ \LpNormOn{ \sum_{ \gamma \in \GamAlR } f_\gamma }{p}{#1} }
\newcommand{\MainThmLHS}{ \MainThmLHSOn{\bR{2}} }
\newcommand{\MainThmRHSOn}[1]{ \LpNormOn{ \bigg( \sum_{ \gamma \in \GamAlR } \abs{f_\gamma}^2 \bigg)^\fhalf }{p}{#1} }
\newcommand{\MainThmRHS}{ \MainThmRHSOn{\bR{2}} }
\newcommand{\MainThmConsFir}{ \delta^{-\sbk{1-\frac{\beta}{2}}\sbk{\fhalf-\frac{1}{p}}} }
\newcommand{\MainThmConsSec}{  \delta^{-\sbk{\frac{1}{2} - \frac{1+\beta}{p}  }  }}

\newcommand{\SharpRect}{ {B_{\delta^{-1}\times\delta^{-\beta}}} }
\newcommand{\SharpRectG}{ {B_{ \abs{\tau}\delta^{-1} \times \abs{\tau}^2\delta^{-\beta}}} }
\newcommand{\SharpRectGR}{   {  B_{\abs{\tau}^2\delta^{-\beta}}  }    }

\pagenumbering{arabic}
\title[Small cap square function and decoupling estimates]{A note on small cap square function and decoupling estimates for the parabola}
\author{Jongchon Kim}
\address{Department of Mathematics\\ City University of Hong Kong\\Hong Kong SAR}
\email{jongckim@cityu.edu.hk}

\author{Liang Wang}
\address{Department of Mathematics\\ City University of Hong Kong\\Hong Kong SAR}
\email{L.Wang@cityu.edu.hk}

\author{Chun Keung Yeung}
\address{Department of Mathematics\\ City University of Hong Kong\\Hong Kong SAR}
\email{ckyeung222-c@my.cityu.edu.hk}


\date{\today}
\begin{abstract}
In this paper, we prove small cap square function and decoupling estimates for the parabola, where the small caps are essentially axis-parallel rectangles of dimensions $\delta\times \delta^\beta$ for $0\leq \beta\leq 1$.
Our estimates complement the known results for $1 \leq \beta \leq 2$ and are sharp up to polylogarithmic factors.
\end{abstract}
\maketitle

\section{ Introduction }
In this paper, we investigate small cap square function and decoupling estimates for the parabola.
Let $\delta$ be a small positive constant and $0\leq\beta\leq 2$. We denote by $\mathcal{N}_{\mathbb{P}_1}(\delta^\beta)$ the $\delta^\beta$-neighborhood of the unit parabola
\[
\mathbb{P}_1 := \bigl\{ (\xi, \xi^2) \in \mathbb{R}^2 : |\xi| \le 1 \bigr\}.
\]
We then consider a disjoint partition $\{ \gamma\}$ of $\mathcal{N}_{\mathbb{P}_1}(\delta^\beta)$, denoted by $\mathcal{P}(\delta)$, where each partition element $\gamma$ is defined as
\[
\gamma := \bigl\{ (\xi_1, \xi_2) \in \mathcal{N}_{\mathbb{P}_1}(\delta^\beta) : \xi_1 \in I  \bigr\},
\]
for an interval $I\subset[-1,1]$ of length $\delta$.
The collection $\mathcal{P}(\delta)$ consists of almost rectangular boxes $\gamma$ with dimensions $\delta \times \delta^\beta$. When $\beta=2$, we call each $\gamma\in \mathcal{P}(\delta)$ a canonical cap. When $1\leq \beta<2$, each $\gamma\in \mathcal{P}(\delta)$ corresponds to a small cap of dimensions $R^{-\alpha}\times R^{-1}$ in the literature (e.g. \cite{DGW20}), where $\delta=R^{-\alpha}$ and $\alpha = \beta^{-1}$.

In this paper, we are primarily concerned with the range $0\leq \beta \leq 1$; in this regime, $\gamma$ is essentially an axis-parallel rectangle with dimensions $\delta\times \delta^\beta$. Let $f$ be an arbitrary function whose Fourier transform is supported in $\mathcal{N}_{\mathbb{P}_1}(\delta^\beta)$.  For any $\gamma\subset \mathbb{R}^2$, we define $f_\gamma$ as the inverse Fourier transform of $\chi_{\gamma}\widehat{f}$, namely
$$
f_\gamma:=(\chi_{\gamma}\widehat{f})^{\vee}.
$$
Let $S_{p,\delta,\beta}$ and $D_{p,\delta,\beta}$ denote the best constants for the following estimates:
\begin{equation*}
    \|f\|_{L^p(\mathbb{R}^2)} \leq S_{p,\delta,\beta} \bigg\|\bigg(\sum_{\gamma\in\mathcal{P}(\delta) }|f_\gamma|^2\bigg)^{1/2}\bigg\|_{L^p(\mathbb{R}^2)}\qquad \text{(Square function estimate)}
\end{equation*}
and
\begin{equation*}
  \|f\|_{L^p(\mathbb{R}^2)}\leq D_{p,\delta,\beta}\bigg(\sum_{\gamma\in\mathcal{P}(\delta)}\|f_\gamma\|_{L^p(\mathbb{R}^2)}^p\bigg)^{\frac{1}{p}}\qquad \text{(Decoupling inequality)},
\end{equation*}
which hold for all $f$ whose Fourier transform is supported in $\mathcal{N}_{\mathbb{P}_1}(\delta^\beta)$.

For the canonical caps, $\beta=2$, the C\'{o}rdoba-Fefferman argument \cite{Fef, Co79} yields the square function estimate
\(S_{4,\delta,2}\lesssim1\). For decoupling inequality, Bourgain and Demeter, in their landmark paper \cite{BD15}, proved that \(D_{6,\delta,2}\leq C_\epsilon\delta^{-\epsilon}\) for any \(\epsilon>0\). Guth, Maldague, and Wang \cite{GMW24} later improved this estimate to \(D_{6,\delta,2}\lesssim |\log\delta|^c\) for some constant \(c>0\).

Motivated by exponential sum estimates over slabs, such as the one appearing in Corollary \ref{cor} below, Demeter, Guth, and Wang \cite{DGW20} introduced the concept of small cap decoupling inequalities. In particular, they proved that for \(1\leq\beta\leq2\) and \( p\geq 2\), the inequality
    \[
         D_{p,\delta,\beta}\leq C_\epsilon\delta^{-\epsilon}\bigg(\delta^{-\bigl(1-\frac{2+\beta}{p}\bigr)}+\delta^{-\bigl(\frac{1}{2}-\frac{1}{p}\bigr)}\bigg)
    \]
    holds for any \(\epsilon>0\). The \(\delta^{-\epsilon}\) loss was improved to \(|\log\delta|^c\) for some constant \(c>0\) by Johnsrude \cite{Johnsrude23}.
As an analogue of small cap decoupling inequalities, Gan \cite{Gan24} established small cap square function estimates: for \(1\leq\beta\leq2\) and \(p\geq2 \),
    \[
         S_{p,\delta,\beta}\leq C_{\epsilon} \delta^{-\epsilon}\bigg(\delta^{-\bigl(1 - \frac{\beta}{2}\bigr)\bigl(\frac{1}{2} - \frac{1}{p}\bigr)} + \delta^{-\bigl(\frac{1}{2} - \frac{2}{p}\bigr)}\bigg).
    \]
Both of the above bounds are sharp up to the \(\delta^{-\epsilon}\) (or \(|\log \delta|^{c}\)) factor.

In this paper, we consider the small cap square function and decoupling estimates for the case \(0\leq\beta\leq1\). Specifically, we establish the following small cap square function estimate:
\begin{theorem}\label{sq func esti thm}
For \(0\leq \beta\leq1\) and \(p\geq2\), we have the following small cap square function estimate:
\[
S_{p,\delta,\beta} \leq C_p|\log \delta|^{c} \left( \MainThmConsFir + \MainThmConsSec \right) ,
\]
where \(C_p\) is a constant depending only on \(p\), and \(c>0\) is a positive constant. The bound is sharp up to the \(|\log \delta|^{c}\) factor.
\end{theorem}
We remark that for \(0 \leq \beta \leq 1\), the inequality
\[
\left(1 - \frac{\beta}{2}\right)\left(\frac{1}{2} - \frac{1}{p}\right) \geq \frac{1}{2} - \frac{1+\beta}{p}
\]
holds if and only if \(p \leq 6\).

Our proof yields Theorem \ref{sq func esti thm} with $c=1+\frac{1}{p}$. Our argument in the proof of Theorem \ref{sq func esti thm} remains valid in the case \(1\leq\beta\leq2\) considered in \cite{Gan24} and can be used to strengthen the result of \cite{Gan24} by replacing the $\delta^{-\epsilon}$ loss with a $|\log \delta|^{c}$ loss.

We further establish the following small cap decoupling inequality:
\begin{theorem}\label{s cap}
For \(0\leq \beta\leq1\) and \(p\geq2\), the following estimate holds:
\begin{equation*}
D_{p,\delta,\beta}\leq C_p|\log \delta|^{c}\bigg(\delta^{-(2-\beta)\bigl(\frac{1}{2}-\frac{1}{p}\bigr)}+\delta^{-(1-\frac{2+\beta}{p})}\bigg),
\end{equation*}
where \(C_p\) is a constant depending only on \(p\), and \(c>0\) is a positive constant. The bound is sharp up to the \(|\log \delta|^{c}\) factor.
\end{theorem}

We remark that for \(0\leq \beta\leq1\), the inequality
\[
(2-\beta)\bigl(\frac{1}{2}-\frac{1}{p}\bigr)\geq 1-\frac{2+\beta}{p}
\]
holds if and only if \( p\leq4\). Theorem \ref{s cap} for the critical exponent $p=4$ is a direct consequence of the small cap decoupling inequality for the case \(\beta=1\) due to Johnsrude \cite{Johnsrude23} and the flat decoupling inequality.

As a corollary of Theorem \ref{s cap}, we obtain the following result.
\begin{corollary}\label{cor}
For \(1 \le \sigma \le 2\) and \(2 \le p <\infty\), the following inequality holds:
\[
\bigg(\int_{[0,1] \times [0, N^{-\sigma}]} \bigg| \sum_{k=1}^N a_k e\bigl(x_1 k + x_2 k^2\bigr) \bigg|^p dx\bigg)^{\frac{1}{p}} \lesssim (\log N)^c \bigg( N^{\frac{\sigma}{2}(1-\frac{4}{p})} + N^{1-\frac{4}{p}} \bigg) \bigg( \sum_{k=1}^N |a_k|^p \bigg)^{\frac{1}{p}},
\]
where \(a_k \in \mathbb{C}\) are arbitrary complex numbers and \(c>0\) is a positive constant. Here, we adopt the standard notation \(e(t) := e^{2\pi i t}\).
\end{corollary}

Corollary \ref{cor} was established in \cite{MO24} with an \(N^\epsilon\) loss for any \(\epsilon>0\) using the small cap decoupling estimate \cite{DGW20} for $\beta=1$, and the result is sharp up to this \(N^\epsilon\) factor (see \cite[Section 1.1]{MO24}). In the present work, Corollary \ref{cor} is obtained as an immediate consequence of Theorem \ref{s cap}. The improvement of the \(N^\epsilon\) loss to the logarithmic factor \( (\log N)^{c}\) is due to the logarithmic improvement in the small cap decoupling estimate \cite{Johnsrude23}. We refer the reader to \cite{DGW20} for sharp bounds for the range $0\leq \sigma \leq 1$.

\subsection*{Organization of the paper}
In Section \ref{S2}, we revise a bilinear reduction argument from \cite{DGW20} by employing the approach introduced in \cite{GMW24}. The proof of Theorem \ref{sq func esti thm} is obtained by estimating the narrow and broad parts. In Section \ref{S3}, we present the proof of Theorem \ref{s cap} and Corollary \ref{cor}. Finally, in Section \ref{S4}, we establish the sharpness of both Theorem \ref{sq func esti thm} and Theorem \ref{s cap}.

\subsection*{Notation}
We summarize below the notations that are frequently used throughout this paper.
\begin{enumerate}
    \item We write $A \lesssim B$ to mean that there exists an absolute constant $C > 0$ such that $A \leq C B$.
Furthermore, we write $A \sim B$ if both $A \lesssim B$ and $B \lesssim A$ hold.
    \item Let \( \delta,\beta,\kappa>0 \). We denote by \( \mathcal{P}(\kappa) \) a disjoint partition \( \{\tau\} \) of \( \mathcal{N}_{\mathbb{P}_1}(\delta^\beta) \), where each element \( \tau \) of the partition is given by
\[
\tau := \bigl\{ (\xi_1, \xi_2) \in \mathbb{R}^2 : \xi_1 \in I,\ |\xi_2 - \xi_1^2| \le \delta^\beta \bigr\},
\]
with \( I\subset[-1,1] \) being an interval of length \( \kappa \). Furthermore, we denote by \( \widetilde{\tau} \) the union of \( \tau \) and its neighboring elements \( \tau^\prime\in\mathcal{P}(\kappa) \).
    \item For a rectangle $T \subset \mathbb{R}^2$ with side lengths $a \times b$, its dual rectangle $T^*$ is defined as the rectangle centered at the origin, with side lengths $a^{-1} \times b^{-1}$, and whose sides are parallel to the corresponding sides of $T$. Additionally, we denote by $w_T$ a smooth function satisfying $w_T \sim 1$ on $T$, decaying rapidly outside $T$, and $\operatorname{supp}\widehat{w_T} \subset T^*$.
\end{enumerate}

 \section{The proof of Theorem \ref{sq func esti thm}}\label{S2}
In this section, we focus on the proof of Theorem \ref{sq func esti thm}.
It suffices to prove the local estimate
\begin{equation}\label{main es1}
  \|f\|_{L^p(T)} \leq C_p |\log\delta|^{c}\Bigl( \MainThmConsFir + \MainThmConsSec \Bigr)  \bigg\|\bigg(\sum_{\gamma\in\mathcal{P}(\delta) }|f_\gamma|^2\bigg)^{1/2}\bigg\|_{L^p(w_T)},
\end{equation}
where $T$ is an axis-parallel rectangle of dimensions $\delta^{-1} \times \delta^{-\beta}$. We obtain Theorem \ref{sq func esti thm} by summing \eqref{main es1} over translates of $T$ partitioning $\mathbb{R}^2$.

\subsection{Broad-Narrow reduction}
In this subsection, we decompose $f$ into broad and narrow parts, which revises the argument in \cite{DGW20} by adopting the idea in \cite{GMW24, Johnsrude23}; this ensures that the small cap square function estimate incurs only a loss of $|\log\delta|^c$. Specifically, let $K=8$ and $m$ be an integer satisfying $K^m\leq\delta^{-\frac{\beta}{2}}<K^{m+1}$. Since $K^{-m} \sim \delta^{\frac{\beta}{2}}$, we henceforth identify $\mathcal{P}(K^{-m})$ with $\mathcal{P}\bigl(\delta^{\frac{\beta}{2}}\bigr)$. Note that each $\theta \in \mathcal{P}\bigl(\delta^{\frac{\beta}{2}}\bigr)$ is a canonical cap of dimensions $\delta^{\beta/2} \times \delta^\beta$. We establish the following estimate:
\begin{equation}\label{main es2}
          \|f\|_{L^p(T)} \leq C_p |\log\delta| \Big( \text{Narrow} + \text{Broad}\Big)^{\frac{1}{p}},
      \end{equation}
where
\begin{equation}\label{narrow}
      \text{Narrow} = \sum_{\theta \in \mathcal{P}\bigl(\delta^{\frac{\beta}{2}}\bigr)} \|f_\theta\|_{L^p(T)}^p,
      \end{equation}
and
\begin{equation}\label{broad}
      \text{Broad} = \sum_{i=0}^{m-1} \sum_{\tau \in \mathcal{P}(K^{-i})} \sum_{\substack{ \tau_1,\tau_2 \in \mathcal{P}(K^{-i-1}) \\ \tau_1,\tau_2 \subset \widetilde{\tau} \\ \text{dist}(\tau_1,\tau_2) \geq  K^{-i-1} }} \LpNormOn{\abs{f_{\tau_1}f_{\tau_2}}^\frac{1}{2}}{p}{T}^p.
      \end{equation}
      Here, for each $\tau\in \mathcal{P}(K^{-i})$, $\widetilde{\tau}$ denotes the union of $\tau$ and its neighboring elements $\tau^\prime\in\mathcal{P}(K^{-i})$.

To prove estimate \eqref{main es2}, we require the following lemmas.
\begin{lemma}\label{inarrow}
    Let $\{a_i\}_{i=1}^n$ be a sequence of complex numbers, and let $C\geq1$ be a positive constant. Then, at least one of the following estimates holds:
    \begin{equation}
        \bigg|\sum_{i=1}^n a_i\bigg|\leq 2C n \max_{\substack{i,j:|i-j|>1}}|a_ia_j|^{\frac{1}{2}},
    \end{equation}
    and
    \begin{equation}
        \bigg|\sum_{i=1}^n a_i\bigg|\leq \left(1+C^{-1}\right)\max_{i}\bigg|\sum_{j:|i-j|\leq1}a_j\bigg|.
    \end{equation}
\end{lemma}

\begin{proof}
    Suppose that the following inequality holds:
    \begin{equation}\label{ina:eq1}
        \bigg|\sum_{i=1}^n a_i\bigg|> \left(1+C^{-1}\right)\max_{i}\bigg|\sum_{j:|i-j|\leq1}a_j\bigg|.
    \end{equation}
    Let $i^*$ and $i^{**}$ be indices such that $|a_{i^*}|=\max_{1\leq i\leq n}|a_i|$ and $|a_{i^{**}}|=\max_{\substack{i:|i-i^*|>1}}|a_i|$. By the triangle inequality, we have
    \begin{equation}\label{ina:eq2}
        \bigg|\sum_{j:|i^*-j|\leq1}a_j\bigg|\geq \bigg|\sum_{i=1}^n a_i\bigg|-n|a_{i^{**}}|.
    \end{equation}
    Combining inequalities \eqref{ina:eq1} and \eqref{ina:eq2}, we derive
    \begin{align*}
        n|a_{i^{**}}|&\geq \bigg|\sum_{i=1}^n a_i\bigg| - \bigg|\sum_{j:|i^*-j|\leq1}a_j\bigg| \\
        &> \bigg|\sum_{i=1}^n a_i\bigg| - \frac{1}{1+C^{-1}}\bigg|\sum_{i=1}^n a_i\bigg| \\
        &= \frac{1}{C+1}\bigg|\sum_{i=1}^n a_i\bigg|.
    \end{align*}
    Note that
    \begin{equation*}
        |a_{i^{**}}|^2 \leq |a_{i^*}a_{i^{**}}| \leq \max_{\substack{i,j:|i-j|>1}}|a_ia_j|,
    \end{equation*}
    which implies $|a_{i^{**}}|\leq \max_{\substack{i,j:|i-j|>1}}|a_ia_j|^{\frac{1}{2}}$. Substituting this into the above inequality, we conclude
    \begin{equation*}
        \bigg|\sum_{i=1}^n a_i\bigg|\leq 2C n\max_{\substack{i,j:|i-j|>1}}|a_ia_j|^{\frac{1}{2}}.
    \end{equation*}
\end{proof}

\begin{lemma}\label{br-na}
The following pointwise bound holds for all $x\in\mathbb{R}^2$ and $p>0$:
\begin{equation}\label{br-na:eq1}
    |f(x)| \leq C_p |\log\delta|\bigg(
        \sum_{\theta \in \mathcal{P}\bigl(\delta^{\frac{\beta}{2}}\bigr)} |f_\theta(x)|^p
        + \sum_{i=0}^{m-1} \sum_{\tau \in \mathcal{P}(K^{-i})}
          \sum_{\substack{ \tau_1,\tau_2 \in \mathcal{P}(K^{-i-1}) \\
                          \tau_1,\tau_2 \subset \widetilde{\tau} \\
                          \operatorname{dist}(\tau_1,\tau_2) \geq  K^{-i-1} }}
          \abs{f_{\tau_1}(x)f_{\tau_2}(x)}^{\frac{p}{2}}
    \bigg)^{\frac{1}{p}},
\end{equation}
where $C_p$ is a constant depending only on $p$ and $c$ is an absolute constant.
\end{lemma}

\begin{proof}
For each $\tau\in \mathcal{P}(K^{-i})$ with $0\leq i\leq m-1$, we first observe the decomposition
\begin{equation*}
    f_{\widetilde{\tau}}(x) = \sum_{\substack{ \tau_1 \in \mathcal{P}(K^{-i-1}) \\ \tau_1\subset {\widetilde{\tau}}}} f_{\tau_1}(x).
\end{equation*}
Applying Lemma \ref{inarrow} to this sum with parameters $n=K=8$ and $C=|\log\delta|$ (note that $C=|\log\delta|\geq1$, which satisfies the hypothesis of Lemma \ref{inarrow}), we conclude that at least one of the following two estimates holds:
\begin{equation}
    |f_{\widetilde{\tau}}(x)|^p \leq \left(16|\log\delta|\right)^p
        \max_{\substack{ \tau_1,\tau_2 \in \mathcal{P}(K^{-i-1}) \\
                        \tau_1,\tau_2 \subset {\widetilde{\tau}} \\
                        \operatorname{dist}(\tau_1,\tau_2) \geq  K^{-i-1} }}
        \abs{f_{\tau_1}(x)f_{\tau_2}(x)}^{\frac{p}{2}},
\end{equation}
and
\begin{equation}
    |f_{\widetilde{\tau}}(x)|^p \leq \left(1+|\log\delta|^{-1}\right)^p
        \max_{\substack{ \tau_1 \in \mathcal{P}(K^{-i-1}) \\
                        \tau_1\subset {\widetilde{\tau}}}} |f_{\widetilde{\tau_1}}(x)|^p.
\end{equation}

Using this for the sum $f = \sum_{\tau_1 \in \mathcal{P}(K^{-1})} f_{\tau_1}$, we obtain
\begin{align}\label{br-na:eq2}
    |f(x)|^p &\leq \left(16|\log\delta|\right)^p
        \max_{\substack{ \tau_1,\tau_2 \in \mathcal{P}(K^{-1})  \\
                        \operatorname{dist}(\tau_1,\tau_2) \geq  K^{-1} }}
        \abs{f_{\tau_1}(x)f_{\tau_2}(x)}^{\frac{p}{2}}
        + \left(1+|\log\delta|^{-1}\right)^p
        \max_{\substack{ \tau_1 \in \mathcal{P}(K^{-1})}} |f_{\widetilde{\tau_1}}(x)|^p.
\end{align}
Next, for any $\tau\in\mathcal{P}(K^{-1})$, we apply the same reasoning to obtain
\begin{align}\label{br-na:eq3}
    |f_{\widetilde{\tau}}(x)|^p &\leq \left(16|\log\delta|\right)^p
        \max_{\substack{ \tau_1,\tau_2 \in \mathcal{P}(K^{-2})  \\
                        \tau_1,\tau_2\subset\widetilde{\tau}\\
                        \operatorname{dist}(\tau_1,\tau_2) \geq  K^{-2} }}
        \abs{f_{\tau_1}(x)f_{\tau_2}(x)}^{\frac{p}{2}}
        + \left(1+|\log\delta|^{-1}\right)^p
        \max_{\substack{ \tau_1 \in \mathcal{P}(K^{-2})\\
                        \tau_1\subset\widetilde{\tau }}} |f_{\widetilde{\tau_1}}(x)|^p.
\end{align}
Combining inequalities \eqref{br-na:eq2} and \eqref{br-na:eq3}, we find
\begin{align*}
|f(x)|^p &\leq \sum_{i=0}^1 \left(1+|\log\delta|^{-1}\right)^{pi} \left(16|\log\delta|\right)^p
    \sum_{\tau\in\mathcal{P}(K^{-i})}
    \max_{\substack{ \tau_1,\tau_2 \in \mathcal{P}(K^{-i-1}) \\
                    \tau_1,\tau_2 \subset {\widetilde{\tau}} \\
                    \operatorname{dist}(\tau_1,\tau_2) \geq  K^{-i-1} }}
    \abs{f_{\tau_1}(x)f_{\tau_2}(x)}^{\frac{p}{2}}\\
&\quad + \left(1+|\log\delta|^{-1}\right)^{2p}
    \max_{\substack{ \tau_1 \in \mathcal{P}(K^{-2}) }} |f_{\widetilde{\tau_1}}(x)|^p.
\end{align*}
Extending this iterative process to $m$ steps, we arrive at
\begin{align*}
|f(x)|^p
& \leq \sum_{i=0}^{m-1} \left(1+|\log\delta|^{-1}\right)^{pi} \left(16|\log\delta|\right)^p
    \sum_{\tau\in\mathcal{P}(K^{-i})}
    \max_{\substack{ \tau_1,\tau_2 \in \mathcal{P}(K^{-i-1}) \\
                    \tau_1,\tau_2 \subset {\widetilde{\tau}} \\
                    \operatorname{dist}(\tau_1,\tau_2) \geq  K^{-i-1} }}
    \abs{f_{\tau_1}(x)f_{\tau_2}(x)}^{\frac{p}{2}}\\
&\quad + \left(1+|\log\delta|^{-1}\right)^{mp}
    \max_{\substack{ \theta \in \mathcal{P}\bigl(\delta^{\frac{\beta}{2}}\bigr) }} |f_{\widetilde{\theta}}(x)|^p\\
& \leq C_p |\log\delta|^{p} \bigg(
    \sum_{\theta \in \mathcal{P}\bigl(\delta^{\frac{\beta}{2}}\bigr)} |f_\theta(x)|^p
    + \sum_{i=0}^{m-1} \sum_{\tau \in \mathcal{P}(K^{-i})}
      \sum_{\substack{ \tau_1,\tau_2 \in \mathcal{P}(K^{-i-1}) \\
                      \tau_1,\tau_2 \subset \widetilde{\tau} \\
                      \operatorname{dist}(\tau_1,\tau_2) \geq  K^{-i-1} }}
      \abs{f_{\tau_1}(x)f_{\tau_2}(x)}^{\frac{p}{2}}
\bigg).
\end{align*}
Here, the final inequality uses the facts that $m< |\log\delta|$ and $\left(1+|\log\delta|^{-1}\right)^{|\log\delta|}$ is bounded by $e$.
\end{proof}
Then, the estimate \eqref{main es2} follows from Lemma \ref{br-na}. Therefore, it suffices to estimate the narrow and broad parts separately.
\subsection{Broad-Narrow analysis}
In this subsection, we aim to establish estimates for the narrow part and the broad part. To this end, we first state the following basic lemmas, which will be essential for our subsequent analysis.
\begin{lemma}[Local $L^2$-orthogonality; see e.g. \cite{Gan24}, Lemma 2]\label{ortho}
Let $\{f_i\}_i$ be a sequence of functions, and let $T\subset\mathbb{R}^2$ be a rectangle such that the family $\{\text{supp} \widehat{f_i} + T^*\}_i$ has finite overlap. Then the estimate
\begin{equation*}
    \int_{\mathbb{R}^2}\bigg|\sum_{i}f_iw_T\bigg|^2\,\mathrm{d}x \lesssim \int_{\mathbb{R}^2}\sum_{i}|f_iw_T|^2\,\mathrm{d}x
\end{equation*}
holds, where $w_T$ denotes a smooth function satisfying $w_T\sim 1$ on $T$, decaying rapidly outside $T$, and $\text{supp}\widehat{w_T}\subset T^*$.
\end{lemma}

\begin{lemma}[Narrow part]\label{Narrow esti}
Let $\theta\in\mathcal{P}(\delta^{\frac{\beta}{2}})$ and  $T$ be an axis-parallel rectangle of dimensions $\delta^{-1} \times \delta^{-\beta}$. Then the following inequality holds for $p\geq 2$:
\[
\LpNormOn{f_\theta}{p}{T} \leq C \MainThmConsFir \bigg\|\bigg(\sum_{\substack{\gamma\in\mathcal{P}(\delta) \\ \gamma \subset \theta}}|f_\gamma|^2\bigg)^{1/2}\bigg\|_{L^p(w_T)}.
\]
\end{lemma}

\begin{proof}

We decompose ${T}$ into a disjoint union of rectangles $B$ with dimensions $\delta^{-\frac{\beta}{2}} \times \delta^{-\beta}$.
Then,
\begin{align*}
   \|f_\theta\|_{L^p({T})}
   &=\bigg(\sum_{B\subset {T}}\|f_\theta\|_{L^p(B)}^p\bigg)^{1/p}\\
   & \lesssim  |B|^{\frac{1}{p}-\frac{1}{2}}\bigg(\sum_{B\subset {T}}\|f_\theta\|_{L^2(w_B)}^p\bigg)^{1/p}
   \qquad\text{(by Bernstein's inequality)}\\
   &\lesssim\delta^{\frac{3\beta}{2}\bigl(\frac{1}{2}-\frac{1}{p}\bigr)}\bigg(\sum_{B\subset {T}}\|f_\theta\|_{L^2(w_B)}^2\bigg)^{1/2}\\
   &\lesssim\delta^{\frac{3\beta}{2}\bigl(\frac{1}{2}-\frac{1}{p}\bigr)}\|f_\theta\|_{L^2(w_{{T}})}.
\end{align*}
By the local $L^2$-orthogonality (Lemma~\ref{ortho}), we further obtain
\begin{align*}
\|f_\theta\|_{L^p(T)}
\lesssim\delta^{\frac{3\beta}{2}\bigl(\frac{1}{2}-\frac{1}{p}\bigr)}\|f_\theta\|_{L^2(w_{{T}})}
\lesssim\delta^{\frac{3\beta}{2}\bigl(\frac{1}{2}-\frac{1}{p}\bigr)}\bigg\|\bigg(\sum_{\substack{\gamma\in\mathcal{P}(\delta) \\ \gamma \subset \theta}}|f_\gamma|^2\bigg)^{1/2}\bigg\|_{L^2(w_{{T}})}.
\end{align*}
Applying H\"older's inequality and using the fact $|T|\sim \delta^{-(1+\beta)}$, we derive
\begin{align*}
\bigg\|\bigg(\sum_{\substack{\gamma\in\mathcal{P}(\delta) \\ \gamma \subset \theta}}|f_\gamma|^2\bigg)^{1/2}\bigg\|_{L^2(w_{{T}})}
&\lesssim|{T}|^{\frac{1}{2}-\frac{1}{p}}\bigg\|\bigg(\sum_{\substack{\gamma\in\mathcal{P}(\delta) \\ \gamma \subset \theta}}|f_\gamma|^2\bigg)^{1/2}\bigg\|_{L^p(w_{{T}})}\\
&\lesssim\delta^{-(1+\beta)\bigl(\frac{1}{2}-\frac{1}{p}\bigr)}\bigg\|\bigg(\sum_{\substack{\gamma\in\mathcal{P}(\delta) \\ \gamma \subset \theta}}|f_\gamma|^2\bigg)^{1/2}\bigg\|_{L^p(w_{T})}.
\end{align*}
Thus, we conclude that
\begin{align*}
  \|f_\theta\|_{L^p(T)}  &\lesssim\delta^{-\bigl(1-\frac{\beta}{2}\bigr)\bigl(\frac{1}{2}-\frac{1}{p}\bigr)}\bigg\|\bigg(\sum_{\substack{\gamma\in\mathcal{P}(\delta) \\ \gamma \subset \theta}}|f_\gamma|^2\bigg)^{1/2}\bigg\|_{L^p(w_{T})}.
\end{align*}
\end{proof}
We next state the following bilinear restriction estimate.

\begin{lemma}\label{bilnear}
Let $1<R$, and let $f_1,f_2$ be functions such that $\text{supp}\widehat{f_1},\text{supp}\widehat{f_2}\subset\mathcal{N}_{\mathbb{P}_1}(R^{-1})$ and $\text{dist}\bigl(\text{supp}\widehat{f_1},\text{supp}\widehat{f_2}\bigr)\gtrsim 1$. Then, we have
\begin{enumerate}
    \item For $2\leq p\leq4$, the estimate
    \begin{equation}\label{bil}
        \|(f_1f_2)^{1/2}\|_{L^p(B_R)}\lesssim R^{\frac{2}{p}-1}\bigl(\|f_1\|_{L^2(w_{B_R})}\|f_2\|_{L^2(w_{B_R})}\bigr)^{1/2}
    \end{equation}
    holds.
    \item For $4\leq p\leq\infty$, the estimate
    \begin{equation}\label{bil1}
        \|(f_1f_2)^{1/2}\|_{L^p(B_R)}\lesssim R^{-1/2}\bigl(\|f_1\|_{L^2(w_{B_R})}\|f_2\|_{L^2(w_{B_R})}\bigr)^{1/2}
    \end{equation}
    holds.
\end{enumerate}
\end{lemma}

\begin{proof}
For $p=2$, estimate \eqref{bil} follows directly from H\"older's inequality. For $p=\infty$, we compute
\begin{align*}
\|(f_1f_2)^{1/2}\|_{L^\infty(B_R)}^2
&\lesssim\bigg\|\widehat{f_1}*\widehat{w_{B(R)}}\bigg\|_{L^1(\mathbb{R}^2)}\bigg\|\widehat{f_2}*\widehat{w_{B(R)}}\bigg\|_{L^1(\mathbb{R}^2)}\\
&\lesssim\bigl|\mathcal{N}_{\mathbb{P}_1}(R^{-1})\bigr|\bigg\|\widehat{f_1}*\widehat{w_{B(R)}}\bigg\|_{L^2(\mathbb{R}^2)}\bigg\|\widehat{f_2}*\widehat{w_{B(R)}}\bigg\|_{L^2(\mathbb{R}^2)}\\
&\lesssim R^{-1}\|f_1\|_{L^2(w_{B_R})}\|f_2\|_{L^2(w_{B_R})}.
\end{align*}
For $p=4$, the result reduces to the standard bilinear restriction estimate (see, e.g., \cite[Lemma 2.4]{Lee04}) . Combining these boundary cases with H\"older's inequality yields estimates \eqref{bil} and \eqref{bil1} for all $2\leq p\leq\infty$.
\end{proof}
\begin{lemma}[Broad part]
\label{Broad esti}
Let $T$ be an axis-parallel rectangle of dimensions $\delta^{-1} \times \delta^{-\beta}$. Fix $\tau\in\mathcal{P}(\Delta)$, where $\Delta=K^{-i}$ for some $0\leq i<m$, and let $\tau_1,\tau_2\in \mathcal{P}(K^{-1}\Delta)$ with $\tau_1,\tau_2\subset\widetilde{\tau}$ and $\text{dist}(\tau_1,\tau_2) \geq  K^{-1}\Delta$. Then, we have
\[ \LpNormOn{\abs{f_{\tau_1}f_{\tau_2}}^\fhalf}{p}{T}  \leqC \sbk{\MainThmConsFir+\MainThmConsSec}  \Big\| \Big(\sum_{ \substack{ \gamma \subset \widetilde{\tau} \\ \gamma \in \mathcal{P}(\delta) }  } \abs{f_\gamma}^2 \Big)^\fhalf\Big\|_{L^p({w_T})}.\]
 \end{lemma}
\begin{proof}
We apply the parabolic rescaling: for a fixed $\tau\in\mathcal{P}(\Delta)$ centered at $(x_1, x_1^2)$, we define the affine transformation $G_\tau$ on $\mathbb{R}^2$ by
\begin{equation*}
    G_\tau(\xi_1,\xi_2):= \begin{bmatrix}
        x_1\\
        x_1^2
        \end{bmatrix} + L_\tau \begin{bmatrix}
        \xi_1\\
        \xi_2
        \end{bmatrix},
\end{equation*} where \( L_\tau = \begin{bmatrix}
        \Delta & 0 \\
        2\Delta x_1 & \Delta^2
        \end{bmatrix} \).
Then $G_\tau^{-1}(\widetilde{\tau})$ coincides with $\mathcal{N}_{\mathbb{P}_1}(\delta^\beta\Delta^{-2})$, and $L_\tau^{\top}({T})$ is a rectangle with dimensions $\delta^{-1}\Delta \times \delta^{-\beta}\Delta^2$, where $L_\tau^\top$ is the transpose of $L_\tau$. We further decompose $L^{\top}_\tau({T})$ into a disjoint union of squares $B$ (rectangles with equal side lengths) with dimensions $\delta^{-\beta}\Delta^2 \times \delta^{-\beta}\Delta^2$, and define
\[
\widehat{g_1}(\xi) := \widehat{f}_{\tau_1}(G_\tau(\xi)) \quad \text{and} \quad \widehat{g_2}(\xi) := \widehat{f}_{\tau_2}(G_\tau(\xi)).
\]
It follows that $\operatorname{supp}\widehat{g_1}, \operatorname{supp}\widehat{g_2} \subset G_\tau^{-1}(\widetilde{\tau})$ and $\operatorname{dist}\bigl(\operatorname{supp}\widehat{g_1},\operatorname{supp}\widehat{g_2}\bigr)>K^{-1}\gtrsim 1$. Moreover, for \(i=1, 2\), \(|f_{\tau_i}(x)| = \abs{\det(L_\tau)} |g_i( L_{\tau}^{\top}x  )|\). Then, we have the following estimate:
\begin{align}\label{bres1}
\LpNormOn{\abs{f_{\tau_1}f_{\tau_2}}^{\frac{1}{2}}}{p}{{T}}
&= \Delta^{3\left(1 - \frac{1}{p}\right)} \Big\|\abs{g_1g_2}^{\frac{1}{2}}\Big\|_{L^p(L_\tau^{\top}({T}))} \\
&\lesssim \Delta^{3\left(1 - \frac{1}{p}\right)} \bigg(\sum_{B\subset L_\tau^{\top}({T})} \Big\|\abs{g_1g_2}^{\frac{1}{2}}\Big\|_{L^p(B)}^p\bigg)^{\frac{1}{p}}. \nonumber
\end{align}

When $2 \le p \le 4$, we apply the bilinear restriction estimate to each $B$, yielding
\begin{equation}\label{bres2}
\Big\|\abs{g_1g_2}^{\frac{1}{2}}\Big\|_{L^p(B)} \lesssim \left(\Delta^2\delta^{-\beta}\right)^{\frac{2}{p} - 1} \bigl(\|g_1\|_{L^2(w_{B})}\|g_2\|_{L^2(w_{B})}\bigr)^{\frac{1}{2}}.
\end{equation}
Combining the estimates \eqref{bres1}, \eqref{bres2} and the inclusion $\ell^1\subset\ell^{\frac{p}{2}}$, we deduce that
\begin{align*}
\LpNormOn{\abs{f_{\tau_1}f_{\tau_2}}^{\frac{1}{2}}}{p}{{T}}
& \lesssim   \Delta^{3\left(1 - \frac{1}{p}\right)} \left(\Delta^2\delta^{-\beta}\right)^{\frac{2}{p} - 1} \bigg(\sum_{B\subset L_\tau^\top(T)} \bigl(\|g_1\|_{L^2(w_{B})}\|g_2\|_{L^2(w_{B})}\bigr)^{\frac{p}{2}} \bigg)^{\frac{1}{p}} \\
& \lesssim   \Delta^{3\left(1 - \frac{1}{p}\right)} \left(\Delta^2\delta^{-\beta}\right)^{\frac{2}{p} - 1} \bigg(\sum_{B\subset L_\tau^\top(T)} \|g_1\|_{L^2(w_{B})}\|g_2\|_{L^2(w_{B})} \bigg)^{\frac{1}{2}}.
\end{align*}
Furthermore, by the Cauchy-Schwarz inequality, we have
\begin{align*}
  \sum_{B\subset L_\tau^\top(T)} \|g_1\|_{L^2(w_{B})}\|g_2\|_{L^2(w_{B})}&\leq
  \bigg(\sum_{B\subset L_\tau^\top(T)} \|g_1\|_{L^2(w_{B})}^2\bigg)^{\frac{1}{2}}\bigg(\sum_{B\subset L_\tau^\top(T)} \|g_2\|_{L^2(w_{B})}^2\bigg)^{\frac{1}{2}}\\
 & \lesssim \|g_1\|_{L^2(w_{L_\tau^\top(T)})}\|g_2\|_{L^2(w_{L_\tau^\top(T)})}\\&\lesssim \Delta^{-3} \|f_{\tau_1}\|_{L^2(w_{T})}\|f_{\tau_2}\|_{L^2(w_{T})}.
\end{align*}
It follows that
\begin{align}\label{bres3}
\nonumber\LpNormOn{\abs{f_{\tau_1}f_{\tau_2}}^{\frac{1}{2}}}{p}{{T}}
 &\lesssim   \Delta^{3\left(1 - \frac{1}{p}\right)} \left(\Delta^2\delta^{-\beta}\right)^{\frac{2}{p} - 1} \Delta^{-\frac{3}{2}} \|f_{\tau_1}\|_{L^2(w_{T})}^{\frac{1}{2}}\|f_{\tau_2}\|_{L^2(w_{T})}^{\frac{1}{2}}\\
 &= \left(\Delta^{-1}\delta^{2\beta}\right)^{\frac{1}{2} - \frac{1}{p}}
          \|f_{\tau_1}\|_{L^2(w_{T})}^{\frac{1}{2}}\|f_{\tau_2}\|_{L^2(w_{T})}^{\frac{1}{2}}.
\end{align}
By Lemma \ref{ortho}, we obtain
\begin{equation}\label{bres4}
     \|f_{\tau_1}\|_{L^2(w_{T})}\|f_{\tau_2}\|_{L^2(w_{T})}\lesssim
     \Big\| \Big(\sum_{\substack{\gamma \subset \widetilde{\tau} \\ \gamma \in \mathcal{P}(\delta)}} \abs{f_\gamma}^2 \Big)^{\frac{1}{2}} \Big\|_{L^2(w_{T})}^2,
\end{equation}
and by H\"older's inequality,
\begin{align}\label{bres5}
  \nonumber\Big\| \Big(\sum_{\substack{\gamma \subset \widetilde{\tau} \\ \gamma \in \mathcal{P}(\delta)}} \abs{f_\gamma}^2 \Big)^{\frac{1}{2}} \Big\|_{L^2(w_{T})}&\lesssim |T|^{\frac{1}{2}-\frac{1}{p}}\Big\| \Big(\sum_{\substack{\gamma \subset \widetilde{\tau} \\ \gamma \in \mathcal{P}(\delta)}} \abs{f_\gamma}^2 \Big)^{\frac{1}{2}} \Big\|_{L^p(w_{T})}
\\&\lesssim \delta^{-(1+\beta)(\frac{1}{2}-\frac{1}{p})}\Big\| \Big(\sum_{\substack{\gamma \subset \widetilde{\tau} \\ \gamma \in \mathcal{P}(\delta)}} \abs{f_\gamma}^2 \Big)^{\frac{1}{2}} \Big\|_{L^p(w_{T})}.
\end{align}
Therefore, combining estimates \eqref{bres3}, \eqref{bres4}, and \eqref{bres5}, we deduce
\begin{align}\label{bres6}
\nonumber
\LpNormOn{\abs{f_{\tau_1}f_{\tau_2}}^{\frac{1}{2}}}{p}{{T}}
\nonumber
&\lesssim \left(\Delta^{-1}\delta^{2\beta}\right)^{\frac{1}{2} - \frac{1}{p}} \delta^{-(1+\beta)\left(\frac{1}{2} - \frac{1}{p}\right)}
  \Big\| \Big(\sum_{\substack{\gamma \subset \widetilde{\tau} \\ \gamma \in \mathcal{P}(\delta)}} \abs{f_\gamma}^2 \Big)^{\frac{1}{2}} \Big\|_{L^p(w_{T})} \\
\nonumber
&= \left(\Delta^{-1}\delta^{\beta - 1}\right)^{\frac{1}{2} - \frac{1}{p}}
  \Big\| \Big(\sum_{\substack{\gamma \subset \widetilde{\tau} \\ \gamma \in \mathcal{P}(\delta)}} \abs{f_\gamma}^2 \Big)^{\frac{1}{2}} \Big\|_{L^p(w_{T})} \\
&\lesssim  \delta^{-\left(1 - \frac{\beta}{2}\right)\left(\frac{1}{2} - \frac{1}{p}\right)}
  \Big\| \Big(\sum_{\substack{\gamma \subset \widetilde{\tau} \\ \gamma \in \mathcal{P}(\delta)}} \abs{f_\gamma}^2 \Big)^{\frac{1}{2}} \Big\|_{L^p(w_{T})},
\end{align}
where the last inequality holds by virtue of $\Delta \ge \delta^{\frac{\beta}{2}}$.

For $4 \le p \le \infty$, we proceed similarly by applying Lemma \ref{bilnear} and Lemma \ref{ortho}, yielding the following estimate:
\begin{align}\label{bres7}
\nonumber\LpNormOn{\abs{f_{\tau_1}f_{\tau_2}}^{\frac{1}{2}}}{p}{{T}}
&\lesssim  \Delta^{3\left(1 - \frac{1}{p}\right)} \left(\Delta^2\delta^{-\beta}\right)^{-\frac{1}{2}} \Delta^{-\frac{3}{2}}
  \Big\| \Big(\sum_{\substack{\gamma \subset \widetilde{\tau}\\ \gamma \in \mathcal{P}(\delta)}} \abs{f_\gamma}^2 \Big)^{\frac{1}{2}} \Big\|_{L^2(w_{{T}})} \\
\nonumber&\lesssim \Delta^{\left(\frac{1}{2} - \frac{3}{p}\right)} \delta^{\frac{\beta}{2}} \delta^{-(1+\beta)\left(\frac{1}{2} - \frac{1}{p}\right)}
  \Big\| \Big(\sum_{\substack{\gamma \subset \widetilde{\tau} \\ \gamma \in \mathcal{P}(\delta)}} \abs{f_\gamma}^2 \Big)^{\frac{1}{2}} \Big\|_{L^p(w_{{T}})} \\
\nonumber&\lesssim \Delta^{\left(\frac{1}{2} - \frac{3}{p}\right)} \delta^{\frac{1+\beta}{p}  - \frac{1}{2}}
  \Big\| \Big(\sum_{\substack{\gamma \subset \widetilde{\tau} \\ \gamma \in \mathcal{P}(\delta)}} \abs{f_\gamma}^2 \Big)^{\frac{1}{2}} \Big\|_{L^p(w_{{T}})} \\
&\lesssim \left(\MainThmConsFir + \MainThmConsSec\right)
  \Big\| \Big(\sum_{\substack{\gamma \subset \widetilde{\tau} \\ \gamma \in \mathcal{P}(\delta)}} \abs{f_\gamma}^2 \Big)^{\frac{1}{2}} \Big\|_{L^p(w_{{T}})},
\end{align}
where the last inequality follows from the condition $\delta^{\frac{\beta}{2}} \le \Delta \le 1$.
Combining estimates \eqref{bres6} and \eqref{bres7}, we obtain the desired conclusion.
\end{proof}

\subsection{Proof of Theorem \ref{sq func esti thm}}
We now complete the proof of Theorem \ref{sq func esti thm}.

Let $T$ be an axis-parallel rectangle of dimensions $\delta^{-1} \times \delta^{-\beta}$. By virtue of inequalities \eqref{main es2}, \eqref{narrow}, and \eqref{broad}, we have
\begin{align*}
\|f\|_{L^p(T)}
&\leq C_p|\log\delta| \bigg( \sum_{\theta \in \mathcal{P}\bigl(\delta^{\frac{\beta}{2}}\bigr)} \|f_\theta\|_{L^p(T)}^p + \sum_{i=0}^{m-1} \sum_{\tau \in \mathcal{P}(K^{-i})} \sum_{\substack{\tau_1,\tau_2 \in \mathcal{P}(K^{-i-1}) \\ \tau_1,\tau_2 \subset \widetilde{\tau} \\ \text{dist}(\tau_1,\tau_2) \geq  K^{-i-1}}} \LpNormOn{\abs{f_{\tau_1}f_{\tau_2}}^{\frac{1}{2}}}{p}{T}^p \bigg)^{\frac{1}{p}}.
\end{align*}
First, applying Lemma \ref{Narrow esti}, we obtain
\begin{align*}
\sum_{\theta \in \mathcal{P}\bigl(\delta^{\frac{\beta}{2}}\bigr)} \|f_\theta\|_{L^p(T)}^p
&\lesssim \delta^{-\left(1 - \frac{\beta}{2}\right)\left(\frac{p}{2} - 1\right)} \sum_{\theta \in \mathcal{P}\bigl(\delta^{\frac{\beta}{2}}\bigr)} \bigg\| \bigg( \sum_{\substack{\gamma \in \mathcal{P}(\delta) \\ \gamma \subset \theta}} |f_\gamma|^2 \bigg)^{\frac{1}{2}} \bigg\|_{L^p(w_T)}^p \\
&\leq \delta^{-\left(1 - \frac{\beta}{2}\right)\left(\frac{p}{2} - 1\right)} \bigg\| \bigg( \sum_{\gamma \in \mathcal{P}(\delta)} |f_\gamma|^2 \bigg)^{\frac{1}{2}} \bigg\|_{L^p(w_T)}^p,
\end{align*}
since $\ell^1 \subset \ell^{\frac{p}{2}}$ when $p\geq 2$. Next, by Lemma \ref{Broad esti}, we derive
\begin{align*}
& \sum_{\tau \in \mathcal{P}(K^{-i})} \sum_{\substack{\tau_1,\tau_2 \in \mathcal{P}(K^{-i-1}) \\ \tau_1,\tau_2 \subset \widetilde{\tau} \\ \text{dist}(\tau_1,\tau_2) \geq  K^{-i-1}}} \LpNormOn{\abs{f_{\tau_1}f_{\tau_2}}^{\frac{1}{2}}}{p}{T}^p \\
&\lesssim \bigg( \delta^{-\left(1 - \frac{\beta}{2}\right)\left(\frac{p}{2} - 1\right)} + \delta^{-\left(\frac{p}{2} - 1 - \beta\right)} \bigg) \sum_{\tau \in \mathcal{P}(K^{-i})} \sum_{\substack{\tau_1,\tau_2 \in \mathcal{P}(K^{-i-1}) \\ \tau_1,\tau_2 \subset \widetilde{\tau} \\ \text{dist}(\tau_1,\tau_2) \geq  K^{-i-1}}}\bigg\| \bigg( \sum_{\substack{\gamma \subset \widetilde{\tau} \\ \gamma \in \mathcal{P}(\delta)}} \abs{f_\gamma}^2 \bigg)^{\frac{1}{2}} \bigg\|_{L^p(w_T)}^p \\
&\lesssim \bigg( \delta^{-\left(1 - \frac{\beta}{2}\right)\left(\frac{p}{2} - 1\right)} + \delta^{-\left(\frac{p}{2} - 1 - \beta\right)} \bigg) \sum_{\tau \in \mathcal{P}(K^{-i})} \bigg\| \bigg( \sum_{\substack{\gamma \subset \widetilde{\tau} \\ \gamma \in \mathcal{P}(\delta)}} \abs{f_\gamma}^2 \bigg)^{\frac{1}{2}} \bigg\|_{L^p(w_T)}^p.
\end{align*}
Here, the final inequality uses the fact that there are at most $\sim 1$ pairs $\tau_1,\tau_2 \subset \widetilde{\tau}$. By using $\ell^1 \subset \ell^{\frac{p}{2}}$ when $p\geq 2$, we have
\begin{align*}
\sum_{\tau \in \mathcal{P}(K^{-i})} \bigg\| \bigg( \sum_{\substack{\gamma \subset \widetilde{\tau} \\ \gamma \in \mathcal{P}(\delta)}} \abs{f_\gamma}^2 \bigg)^{\frac{1}{2}} \bigg\|_{L^p(w_T)}^p \lesssim \bigg\| \bigg( \sum_{\gamma \in \mathcal{P}(\delta)} |f_\gamma|^2 \bigg)^{\frac{1}{2}} \bigg\|_{L^p(w_T)}^p.
\end{align*}
Combining the above estimates and the fact that $m < |\log\delta|$, we conclude that
\begin{equation*}
\|f\|_{L^p(T)}^p
\leq C_p |\log\delta|^{p+1} \bigg( \delta^{-\left(1 - \frac{\beta}{2}\right)\left(\frac{p}{2} - 1\right)} + \delta^{-\left(\frac{p}{2} - 1 - \beta\right)} \bigg) \bigg\| \bigg( \sum_{\gamma \in \mathcal{P}(\delta)} |f_\gamma|^2 \bigg)^{\frac{1}{2}} \bigg\|_{L^p(w_T)}^p,
\end{equation*}
which verifies the local estimate \eqref{main es1} with $c=1+\frac{1}{p}$.

\section{The proof of Theorem \ref{s cap} and Corollary \ref{cor}}\label{S3}
\subsection{The proof of Theorem \ref{s cap}}
Recall that aim is to show that for $\beta\in[0,1]$ and $p\geq2$, the following inequality holds:
\begin{equation}
    \|f\|_{L^p(\mathbb{R}^2)}\leq C_p|\log\delta|^{c}\bigg(\delta^{-(2-\beta)\bigl(\frac{1}{2}-\frac{1}{p}\bigr)}+\delta^{-(1-\frac{2+\beta}{p})}\bigg)\bigg(\sum_{\gamma\in\mathcal{P}(\delta)}\|f_\gamma\|_{L^p(\mathbb{R}^2)}^p\bigg)^{\frac{1}{p}}.
\end{equation}
We shall verify this by interpolating $L^2,\, L^4$ and $L^\infty$ estimates (see, e.g., \cite[Remark 3.4]{Lee20}). More precisely, we use the interpolation argument for equivalent small cap decoupling estimates, where $f_\gamma$ is defined using a smooth frequency cutoff.

By Plancherel's theorem and the triangle inequality, we first have
\begin{equation*}\label{dec-es1}
  \|f\|_{L^2(\mathbb{R}^2)} \lesssim  \bigg(\sum_{\gamma\in\mathcal{P}(\delta)}\|f_\gamma\|_{L^2(\mathbb{R}^2)}^2\bigg)^{\frac{1}{2}}
\end{equation*}
and
\begin{equation*}\label{dec-es2}
  \|f\|_{L^\infty(\mathbb{R}^2)}\leq \delta^{-1}\sup_{\gamma\in\mathcal{P}(\delta)}\|f_\gamma\|_{L^{\infty}(\mathbb{R}^2)}.
\end{equation*}
Hence, by interpolation, it suffices to prove that
\begin{equation*}\label{dec-es3}
    \|f\|_{L^4(\mathbb{R}^2)}\lesssim|\log\delta|^{c}\delta^{-\frac{2-\beta}{4}}\bigg(\sum_{\gamma\in \mathcal{P}(\delta)}\|f_\gamma\|_{L^4(\mathbb{R}^2)}^4\bigg)^{\frac{1}{4}}.
\end{equation*}

To this end, we apply the small cap decoupling inequality in the case $\beta = 1$ \cite[Theorem 1.1]{Johnsrude23}, together with the flat decoupling inequality (see, e.g., \cite[Proposition 2.4]{DGW20}). This leads to the following deduction.
\begin{align*}
\|f\|_{L^4(\mathbb{R}^2)}&\lesssim|\log\delta|^{c}\delta^{-\frac{\beta}{4}} \bigg(\sum_{\theta\in\mathcal{P}(\delta^{\beta})}\|f_\theta\|_{L^4(\mathbb{R}^2)}^4\bigg)^{\frac{1}{4}}  \\
&\lesssim|\log\delta|^{c} \delta^{-\frac{\beta}{4}} \delta^{\frac{\beta-1}{2}} \bigg(\sum_{\theta\in\mathcal{P}(\delta^{\beta})}\sum_{\substack{\gamma\subset \theta\\\gamma\in\mathcal{P}(\delta)}}\|f_\gamma\|_{L^4(\mathbb{R}^2)}^4\bigg)^{\frac{1}{4}}\\
&\lesssim|\log\delta|^{c}\delta^{-\frac{2-\beta}{4}}\bigg(\sum_{\gamma\in \mathcal{P}(\delta)}\|f_\gamma\|_{L^4(\mathbb{R}^2)}^4\bigg)^{\frac{1}{4}}.
\end{align*}

\subsection{The proof of Corollary \ref{cor}}
    A change of variables yields
    \begin{equation}\label{co es1}
     N^3\int_{[0,1]\times[0,\frac{1}{N^\sigma}]}\bigg|\sum_{k=1}^Na_ke\bigl(x_1k+x_2k^2\bigr)\bigg|^pdx=\int_{[0,N]\times[0,N^{2-\sigma}]}\bigg|\sum_{k=1}^Na_ke\bigl(x_1\cdot\frac{k}{N}+x_2\cdot\frac{k^2}{N^2}\bigr)\bigg|^pdx.
    \end{equation}
    Let $\phi$ be a Schwartz function on $\mathbb{R}^2$ satisfying $|\phi|\sim1$ on $[0,N]\times[0,N^{2-\sigma}]$ and $\text{supp} \widehat{\phi}\subset [0,N^{-1}]\times[0,N^{\sigma-2}]$.
    Applying Theorem \ref{s cap} with $\delta=N^{-1}$ and $\beta=2-\sigma$, we obtain
    \begin{align}\label{co es2}
    \nonumber&\int_{[0,N]\times[0,N^{2-\sigma}]}\bigg|\sum_{k=1}^Na_ke\bigl(x_1\cdot\frac{k}{N}+x_2\cdot\frac{k^2}{N^2}\bigr)\bigg|^pdx\\
\nonumber&\lesssim\int_{\mathbb{R}^2}\bigg|\sum_{k=1}^N\phi(x)a_ke\bigl(x_1\cdot\frac{k}{N}+x_2\cdot\frac{k^2}{N^2}\bigr)\bigg|^pdx\\
\nonumber&\leq C_p (\log N)^{pc}\bigg(N^{\frac{\sigma}{2}(p-2)}+N^{p-4+\sigma}\bigg)\bigg(\sum_{k=1}^N\int_{\mathbb{R}^2}\bigg|\phi(x)a_ke\bigl(x_1\cdot\frac{k}{N}+x_2\cdot\frac{k^2}{N^2}\bigr)\bigg|^pdx\bigg)\\
\nonumber&\leq C_p (\log N)^{pc}\bigg(N^{\frac{\sigma}{2}(p-2)}+N^{p-4+\sigma}\bigg)\bigg(\sum_{k=1}^N|a_k|^p\int_{\mathbb{R}^2}|\phi(x)|^pdx\bigg)\\
&\leq C_p (\log N)^{pc} N^{3-\sigma}\bigg(N^{\frac{\sigma}{2}(p-2)}+N^{p-4+\sigma}\bigg)\bigg(\sum_{k=1}^N|a_k|^p\bigg).
    \end{align}
Combining inequalities \eqref{co es1} and \eqref{co es2}, we conclude that
\begin{equation}
  \int_{[0,1]\times[0,\frac{1}{N^\sigma}]}\bigg|\sum_{k=1}^Na_ke\bigl(x_1k+x_2k^2\bigr)\bigg|^pdx\leq C_p (\log N)^{pc}\bigg(N^{\frac{\sigma}{2}(p-4)}+N^{p-4}\bigg)\bigg(\sum_{k=1}^N|a_k|^p\bigg).
\end{equation}

\section{ Sharpness }\label{S4}
We now demonstrate the sharpness of Theorem \ref{sq func esti thm} and Theorem \ref{s cap}. We consider two standard examples: the constructive interference and block examples (see e.g. \cite{FGM23}).

\subsection*{Example 1}
For each $\gamma \in \mathcal{P}(\delta)$, let $\widehat{f_\gamma}$ be a bump function supported on $\gamma$, and define
\[ f = \sum_{\gamma \in \mathcal{P}(\delta)} f_\gamma. \]
On the one hand, observe that $|f(x)| \gtrsim \delta^\beta$ when $x$ is near the origin, which implies that
\begin{equation}
    \|f\|_{L^p(\mathbb{R}^2)} \gtrsim \delta^\beta \cdot |B(0,1)| = \delta^\beta.
\end{equation}
On the other hand, we can find, via direct computation, that $|f_\gamma| \sim |\gamma| = \delta^{1+\beta}$ on the dual rectangle $\gamma^*$, which can be identified as $\gamma^*= [-\delta^{-1},\delta^{-1}] \times[-\delta^{-\beta},\delta^{-\beta}]$ for all $\gamma$, and $|f_\gamma|$ decays rapidly outside of $\gamma^*$.
Thus, we have
\begin{align*}
    \bigg(\sum_{\gamma \in \mathcal{P}(\delta)} \|f_\gamma\|_{L^p(\mathbb{R}^2)}^p\bigg)^{\frac{1}{p}} &\lesssim \left(\delta^{-1} \cdot \delta^{(1+\beta)(p-1)}\right)^{\frac{1}{p}} = \delta^{-\frac{1}{p} + (1+\beta)\left(1 - \frac{1}{p}\right)}, \\
  \MainThmRHS &\leqC |\gamma| \cdot \LpNormOn{\bigg( \sum_{ \gamma \in \GamAlR } 1\bigg)^{\frac{1}{2}}}{p}{ \gamma^* } \leqC \delta^{1+\beta - \frac{1}{2}} \cdot \delta^{-\frac{1+\beta}{p}} = \delta^{\frac{1}{2} + \beta - \frac{1+\beta}{p}}.
\end{align*}
Therefore, we deduce that
\begin{equation*}
    \frac{\|f\|_{L^p(\mathbb{R}^2)}}{\bigg(\sum_{\gamma \in \mathcal{P}(\delta)} \|f_\gamma\|_{L^p(\mathbb{R}^2)}^p\bigg)^{\frac{1}{p}}} \gtrsim \frac{\delta^\beta}{\delta^{1+\beta - \frac{2+\beta}{p} }} = \delta^{-\left(1 - \frac{2+\beta}{p}\right)},
\end{equation*}
and
\[ \frac{\|f\|_{L^p(\mathbb{R}^2)}}{\MainThmRHS} \geqC \frac{\delta^\beta}{\delta^{\frac{1}{2} + \beta - \frac{1+\beta}{p}}} = \MainThmConsSec. \]

\subsection*{Example 2} Fix a canonical cap \( \theta \in \GamHfR \). For each \( \gamma \in \mathcal{P}(\delta) \), let \( \widehat{f_\gamma} \) be a bump function supported on \( \gamma \) if \( \gamma \subset \theta \), and set \( \widehat{f_\gamma} = 0 \) otherwise. We then define
\[ f = \sum_{\gamma\in\mathcal{P}(\delta)}f_\gamma. \]
Then we have \(| f_\gamma| \sim |\gamma| = \delta^{1+\beta} \) on \( \gamma^* \) and decays rapidly outside the dual rectangle \( \gamma^* \), additionally, \( |f| \sim |\theta| = \delta^{\frac{3\beta}{2}} \) on \( \theta^* \) and decays rapidly outside the dual rectangle \( \theta^* \).
We now have
\begin{equation*}
    \|f\|_{L^p(\mathbb{R}^2)} \geqC \ |\theta| \cdot |\theta^*|^{\frac{1}{p}} = \delta^{\frac{3\beta}{2}\left(1-\frac{1}{p}\right)},
\end{equation*}
\begin{align*}
    \bigg(\sum_{\gamma \in \mathcal{P}(\delta)} \|f_\gamma\|_{L^p(\mathbb{R}^2)}^p\bigg)^{\frac{1}{p}} \lesssim \left(\delta^{\frac{\beta}{2}-1} \cdot \delta^{(1+\beta)(p-1)}\right)^{\frac{1}{p}} = \delta^{\beta\left(1-\frac{1}{2p}\right)+1-\frac{2}{p}},
\end{align*}
and
\[ \MainThmRHS \leqC |\gamma| \cdot \LpNormOn{\bigg(\sum_{ \gamma \subset \theta } 1\bigg)^{\frac{1}{2}}}{p}{ \gamma^* } = \delta^{\frac{1}{2}+\frac{5\beta}{4}-\frac{1+\beta}{p}}. \]
Therefore, we deduce that
\begin{equation*}
    \frac{\|f\|_{L^p(\mathbb{R}^2)}}{\bigg(\sum_{\gamma \in \mathcal{P}(\delta)} \|f_\gamma\|_{L^p(\mathbb{R}^2)}^p\bigg)^{\frac{1}{p}}} \gtrsim \frac{\delta^{\frac{3\beta}{2}\left(1-\frac{1}{p}\right)}}{\delta^{\beta\left(1-\frac{1}{2p}\right)+1-\frac{2}{p}}} = \delta^{-(2-\beta)\left(\frac{1}{2} - \frac{1}{p}\right)},
\end{equation*}
and
\[ \frac{\|f\|_{L^p(\mathbb{R}^2)}}{\MainThmRHS} \gtrsim \frac{\delta^{\frac{3\beta}{2}\left(1-\frac{1}{p}\right)}}{\delta^{\frac{1}{2}+\frac{5\beta}{4}-\frac{1+\beta}{p}}} = \delta^{-(1-\frac{\beta}{2})\left(\frac{1}{2} - \frac{1}{p}\right)}. \]

\subsection*{Acknowledgements}
The authors were supported in part by grants from the Research Grants Council of the Hong Kong Administrative Region, China (Project No. CityU 21309222 and CityU 11308924).
\bibliographystyle{amsplain}

\end{document}